\documentclass[12pt]{amsart}
\usepackage{graphicx} 
\usepackage{amsthm}
\usepackage{amssymb}
\usepackage{mathtools}
\usepackage{MnSymbol}
\usepackage{xcolor}
\usepackage{bbm}
\usepackage[normalem]{ulem}
\DeclareMathAlphabet{\mymathbb}{U}{bbold}{m}{n}
\usepackage{dsfont}
\usepackage{hyperref}

\usepackage{amssymb}

\usepackage{enumitem}

\makeatletter
\@namedef{subjclassname@2020}{%
  \textup{2020} Mathematics Subject Classification}
\makeatother

\usepackage[T1]{fontenc}

\newtheorem{theorem}{Theorem}[section]

\newtheorem{lemma}[theorem]{Lemma}

\newtheorem{fact}[theorem]{Fact}

\theoremstyle{definition}
\newtheorem{definition}[theorem]{Definition}

\numberwithin{equation}{section}

\begin{document}


\baselineskip=17pt


\title{Projectional entropy for actions of amenable groups}

\author[M. Prusik]{Michał Prusik}
\address{Faculty of Pure and Applied Mathematics\\ Wroclaw University of Science and Technology\\
Wybrzeże Wyspiańskiego 27\\
50-370 Wrocław Poland}
\email{michal.prusik@pwr.edu.pl}

\date{}

\begin{abstract}
In the current paper we attempt to transfer the notion of the projectional entropy, originally defined for multidimensional subshifts, to the case of actions of amenable groups. The main theorem states that if a system is strongly irreducible the equality of the entropy and the projectional entropy implies that the system has a product-like structure.
\end{abstract}



\maketitle

	\section{Introduction}
	In ``Projectional Entropy in Higher Dimensional Shifts of Finite Type'' (\cite{JKM}) A. Johnson, S. Kass and K. Madden introduced an object called the \textit{projectional entropy}. The authors {showed several} basic properties {of this entropy} and {proved} a theorem {relating} this notion {to} the structure of the subshift {in the case of two-dimensional subshifts}. In 2010 M. Schraudner {extended} their results in {\cite{S} to higher dimensional subshifts, by adding an assumption of mixing type}. {In the current paper we aim} to generalize the notion of the projectional entropy for actions of amenable groups. 
	
	The entropy is one of the main objects of interest in the theory of dynamical systems. For the case of symbolic dynamics it gives us an information about the rate of exponential growth of the number of configurations appearing in the system. That growth {is measured when the domain of a configuration expands simultaneously} in every ``direction'' of the space. The main intuition behind the projectional entropy is to look only in some directions. Namely, we restrict the elements of the subshift $X\subset \mathcal A ^{\mathbb Z^d}$ (where $\mathcal A$ is a fixed alphabet) to a chosen \textit{sublattice} $L$, where a sublattice is a subgroup of $\mathbb Z^d$ isomorphic to $\mathbb Z^r$ for some $r<d$, which has a \textit{complementary} sublattice $L'$ (namely the quotient group $\mathbb Z^d/L$) isomorphic to $\mathbb Z^{(d-r)}$. So obtained restriction $X_L$ is treated as an $r$-dimensional subshift and its entropy is called the \textit{$L$-projectional entropy of $X$}. By comparing this {quantity} with the entropy of the original space $X$ we get some information about a ``product-like'' structure of $X$.
	
    It turns out that one does not need to limit only to sublattices. Consider $L=\mathrm{span}\{(2,0),(0,2)\}\subset \mathbb Z^2$ (the lattice of vectors with both even coordinates). One can easily see, that in this case the quotient group is isomorphic to $\mathbb Z_2\times \mathbb Z_2$, so $L$ is not a sublattice. However the restriction $X_L$ is still a subshift of $\mathcal A^{\mathbb Z^2}$ and at least some of the results of \cite{S} hold in this case. 
    This led us to a hypothesis, that the projectional entropy may be well defined for a restriction to any normal subgroup. In the current paper we attempt to transfer the results of \cite{JKM} and \cite{S} to the case of actions of amenable groups.
    The main theorem states that if a certain mixing condition holds, namely the system is strongly irreducible, the equality of the entropy and the projectional entropy with respect to a certain normal subgroup implies that the system has a product-like structure.

	\section{Basic notions}
	To simplify the notation, for a countable set $A$ let $\mathcal F(A)$ denote the family of all finite and nonempty subsets of $A$.
	\begin{definition}
		Let $G$ be a countable group. A sequence $F_n\in\mathcal F(G)$ is a \textit{F{\o}lner sequence}, if it satisfies the following condition:
		$$\forall_{g\in G}~~\lim_{n\to \infty}\frac{|gF_n\bigtriangleup F_n|}{|F_n|}=0,$$
		where $gF=\{gf:f\in F\}$ for $g\in G, F\subset G$.
		
		We call $G$ \textit{amenable}, if there exists a F{\o}lner sequence in $G$.
	\end{definition}
	It is obvious, that if $(F_n)$ is a F{\o}lner sequence, then for any nonempty, finite $T\subset G$ the term $\frac{|TF_n\bigtriangleup F_n|}{|F_n|}$ converges to $0$ (for $A,B\subset G$ by $AB$ we denote the set $\{ab:a\in A, b\in B\}$). Also the following holds. 	
        \begin{lemma}
		\label{lematopowiekszaniuCF}
		Let $(F_n)$ be a Følner sequence in $G$ and let $B$ be a nonempty finite subset of $G$. Then $\tilde F_n:=BF_n$ is also a Følner sequence in $G$.
	\end{lemma}
	\begin{proof}
		Fix $g\in G$. From triangle inequality for symmetric difference and the fact that $|BF_n|\geqslant |F_n|$ we obtain:
		$$\frac{|g\tilde F_n \bigtriangleup \tilde F_n|}{|\tilde F_n|}\leqslant \frac{|gBF_n\bigtriangleup F_n|}{|F_n|} + \frac{|F_n\bigtriangleup BF_n|}{|F_n|}.$$
		Since $B$ (and so $gB$) is nonempty and finite, and $(F_n)$ is a Følner sequence, both terms on the right converge to zero.
	\end{proof}
	
	The fact below is widely known.
	\begin{fact}
		\label{fact_amenable_nor_subgr_i_quotient}
		Let $G$ be a countable group and $H$ be its normal subgroup. Then
		$$G \textrm{ is amenable }\iff H\textrm{ and } G/H \textrm{ are amenable.}$$
	\end{fact}
 
	\begin{definition}
		Let $G$ be a countable amenable group and let $\mathcal{A}$ be a finite set with the  discrete topology. We define the \textit{shift action} of $G$ on $\mathcal A^G$ as follows:
		$$\forall_{x\in \mathcal{A}^G} ~ \forall_{g,h\in G} ~~(gx)(h)=x(hg).$$
		
		The \textit{full $G$-shift} (or just the \textit{full shift}, if $G$ is known from the context) is the product space $\mathcal{A}^G$ equipped with the product topology and the shift action of $G$. We will call the set $\mathcal{A}$ the \textit{alphabet}, and its elements will be \textit{symbols}.
	\end{definition}

    Each element of the full $G$-shift is in fact a function $x:G\to \mathcal{A}$ thus it is reasonable to consider its restriction $x|_F$ to a set $F\subset G$.
		Let us notice that the full $G$-shift is, by the Tikhonov's theorem, a compact topological space. It is well known that the the full shift is metrizable and that a sequence $x_n\in\mathcal A^G$ converges to $x\in\mathcal A^G$ if and only if for any $g\in G$ we have $x_n(g)=x(g)$ for $n$ bigger than some $N_g$.
	
	We say that $Y\subset \mathcal A^G$ is $G$\textit{-invariant} if it is invariant under the shift action of $G$, i.e. $gY=\{gy: y\in Y\}=Y$ for any $g\in G$.
	\begin{definition}
		A nonempty, compact and $G$-invariant set $X\subset \mathcal A^G$ equipped with the shift action of $G$ will be called a \textit{$G$-subshift} or just a \textit{subshift} if $G$ is known from the context.
	\end{definition}
	
	\begin{definition}
		Let $F\in \mathcal F(G)$. A \textit{configuration on $F$} is an element of $\mathcal A^F$. We say that a configuration $\mathsf{C}$ on $F$ \textit{appears} in $x\in\mathcal A^G$ if $x|_{F}=\mathsf C$.
		
		For a subshift $X$ by $\mathcal L_F(X)$ we denote the set of all configurations on $F$ that appear in the elements of $X$. The set $\mathcal L(X)=\bigcup _{F\in \mathcal F(G)}\mathcal L_F(X)$ is called the \textit{language} of $X$.
	\end{definition}
	Similarly to the classical case, one obtains an equivalent definition of a subshift by forbidding configurations. 
	\begin{theorem}
		$X\subset \mathcal A^G$ is a subshift if and only if there exists $\mathcal N\subset \bigcup_{F\in \mathcal{F}(G)}\mathcal A^F$ (a set of configurations) such that
		$$X=\{x\in\mathcal A^G: \forall_{g\in G}\forall_{\mathsf C\in \mathcal N} ~\mathsf C \textrm{ does not appear in } gx\}.$$
	\end{theorem}
	
	We will use the following simple fact.
	\begin{fact}
		\label{fact_zawieranie_jezykow}
		If $Y$ is a proper subshift of $X$ (that is $Y\subsetneq X$), then $\mathcal L(Y)\subsetneq\mathcal L(X)$.
	\end{fact}
		
	Finally, we define the entropy of a $G$-subshift. The theorem below is a consequence of the Ornstein-Weiss Lemma (see \cite{LW}, Theorem 6.1).
	\begin{theorem}
		Let $X$ be a $G$-subshift. For any F{\o}lner sequence $(F_n)$ in $G$ the limit 
		$$\lim_{n\to \infty}\frac{1}{|F_n|}\log\left|\mathcal L_{F_n}(X)\right|$$
		exists, is finite, and does not depend on the choice of $(F_n)$. The common value $h(X)$ of the limit is called the \textup{entropy} of $X$.
	\end{theorem}
	
	We will use the following result, which was introduced in \cite{DFR} and is called the ``{infimum rule}''.
	\begin{theorem}
		\label{zasadainfimum}
		The following equality holds:
		$$h(X)=\inf_{F\in \mathcal F(G)}\frac{1}{|F|}\log\left|\mathcal{L}_F(X)\right|.$$
	\end{theorem}
	
	\section{The results}
	
	Throughout this section $G$ is a fixed countable amenable group, $\mathcal{A}$ is a finite alphabet and $X\subset \mathcal{A}^G$ is a $G$-subshift. Additionally, let $H$ denote a fixed normal subgroup of $G$. 
	\begin{definition}
		A \textit{projection of $X$ onto} $H$ is a subset $X_H$ of $\mathcal A^H$ defined as follows:
		$$X_H=\left\{x'\in\mathcal A^H: \exists _{x\in X} \textrm{~}x'= x|_H\right\}.$$
	\end{definition}
	By Fact \ref{fact_amenable_nor_subgr_i_quotient} we know that $H$ is amenable, so we can formulate the following theorem.
	\begin{theorem}
		\label{th:X_Htoshift}
		$X_H$ is an $H$-subshift.
	\end{theorem}
	\begin{proof}
		\textbf{$H$-invariance}. Let $h\in H$ and $x'\in X_H$. By the definition, there exists $x\in X$ such that $x|_H=x'$. For any $\tilde h\in H$ we have:
		$$hx'(\tilde h)=x'(\tilde hh)=x(\tilde hh)=hx(\tilde h),$$
		hence $hx'=(hx)|_H$. The subshift $X$ is $G$-invariant, so $hx \in X$, and thus $hx'\in X_H$. Therefore we get:
		$$\forall_{h\in H}~~hX_H\subset X_H.$$
		But that means that for any $h\in H$ also $h^{-1}X_H\subset X_H$. Hence $X_H=h(h^{-1}X_H)\subset hX_H$, for any $h\in H$.
		
		\textbf{Compactness}. Let $(x'_n)_{n\geqslant1}$ be a sequence in $X_H$. We know that for any $n$ there exists $x_n\in X$ such that $x'_n=x_n|_H$. By the compactness of $X$ there is a subsequence $x_{n_k}$ convergent to some $x\in X$. Then for each $g\in G$ $x_{n_k}(g)$ eventually equals $x(g)$, in particular it holds for each $g\in H$. This gives us the convergence $x'_{n_k}\rightarrow x|_H\in X_H$.
	\end{proof}

    The following definition introduces our main object of interest.
	\begin{definition}
		The entropy of $X_H$ (as a subshift of $\mathcal A^H$) is the \textit{$H$-projectional entropy of $X$}.
	\end{definition}
	Below we present a method of constructing a subshift of $\mathcal A^G$ from $X_H$ in a ``product-like'' way.
	\begin{definition}
		A \textit{transversal for $H$} is a subset $M$ of $G$ such that for any coset $gH$ we have $|M\cap gH|=1$ (i.e. $M$ consists of representatives of the cosets).
	\end{definition}
    For a transversal $M$ for $H$ we define a function $\varphi_M:\left(X_H\right)^{G/H}\rightarrow \mathcal A^G$ in the following way. For any $g\in G$ there exists exactly one $h\in H$ and exactly one $m\in M$ such that $g=hm$. We take
		$$\varphi_M\left((x^\gamma)_{\gamma\in G/H}\right)(g)=x^{Hm}(h),$$
		where $x^\gamma \in X_H$ for any coset $\gamma\in G/H$.
	\begin{theorem}
		~
		\begin{enumerate} 
			\item For each transversal $M$ the function $\varphi_M$ is an injection.
			\item For any two transversals $M_1, M_2$ we have $$\varphi_{M_1}\left(\left(X_H\right)^{G/H}\right)=\varphi_{M_2}\left(\left(X_H\right)^{G/H}\right).$$
   We will denote the common image by $X_H^{G/H}$.
			\item For each transversal $M$ the function $\varphi_M$ is continuous with respect to the product topology on $\left(X_H\right)^{G/H}$.
			\item $X_H^{G/H}$ is a $G$-subshift.
		\end{enumerate}
	\end{theorem}
	Let us note that by $\left(X_H\right)^{G/H}$ we denote the Cartesian power of the set $X_H$, while $X_H^{G/H}$ is a subset of $\mathcal A^G$.
	
	\begin{proof}
		We will use the notation $\left(x^\gamma\right)=\left(x^\gamma\right)_{\gamma\in G/H}$.
		\begin{enumerate}
			\item Let $\left(x^\gamma\right)\neq\left(y^\gamma\right)$. Then there exists $\gamma_0\in G/H$ such that $x^{\gamma_0}\neq y^{\gamma_0}$, which means that there exists $h_0\in H$ such that $x^{\gamma_0}(h_0)\neq y^{\gamma_0}(h_0)$. Let $m_0\in M\cap \gamma_0$ (in fact $M\cap \gamma_0=\{m_0\}$). Then we have
			$$\varphi_M\left(\left(x^\gamma\right)\right)(h_0m_0)=x^{\gamma_0}(h_0)\neq y^{\gamma_0}(h_0)=\varphi_M\left(\left(y^\gamma\right)\right)(h_0m_0),$$
			so $\varphi_M\left(\left(x^\gamma\right)\right)\neq\varphi_M\left(\left(y^\gamma\right)\right)$. 
			
			\item Let $M_1$, $M_2$ be transversals for $H$. By the symmetry it suffices to show that
			$$\varphi_{M_1}\left(\left(X_H\right)^{G/H}\right)\subset\varphi_{M_2}\left(\left(X_H\right)^{G/H}\right).$$
			Fix $\left(x^\gamma\right)\in\left(X_H\right)^{G/H}$. For any $\gamma\in G/H$ let $m_i^\gamma$ be the only element of $ M_i\cap\gamma$, $i=1,2$, and take $y^\gamma=m_2^{\gamma}\left(m_1^{\gamma}\right)^{-1}x^\gamma\in \mathcal A^H$. It is well defined, because $m_2^{\gamma}\left(m_1^{\gamma}\right)^{-1}\in H$. We have $y^\gamma\in X_H$ for any $\gamma$, since $X_H$ is $H$-invariant.
			
			Fix $g\in G$ and take $\gamma_0=gH$. There exists exactly one pair $h_1,h_2\in H$ such that $g=h_1m_1^{\gamma_0}=h_2m_2^{\gamma_0}$. Note that then $h_1=h_2m_2^{\gamma_0}\left(m_1^{\gamma_0}\right)^{-1}$. We have
			$$\varphi_{M_1}\left(\left(x^\gamma\right)\right)(g)=x^{\gamma_0}(h_1)=x^{\gamma_0}\left(h_2m_2^{\gamma_0}\left(m_1^{\gamma_0}\right)^{-1}\right)$$
			$$=m_2^{\gamma_0}\left(m_1^{\gamma_0}\right)^{-1}x^{\gamma_0}(h_2)=y^{\gamma_0}(h_2)=\varphi_{M_2}\left(\left(y^\gamma\right)\right)(g).$$
			Since $g$ and $(x^\gamma)$ were arbitrary, we get that			
			$$
            \forall_{(x^\gamma)\in (X_H)^{G/H}} ~ \exists_{(y^\gamma)\in(X_H)^{G/H}}~~\varphi_{M_1}((x^\gamma))=\varphi_{M_2}((y^\gamma)),
            $$
			thus $\varphi_{M_1}\left(\left(X_H\right)^{G/H}\right)\subset\varphi_{M_2}\left(\left(X_H\right)^{G/H}\right).$
			
			\item We skip the proof of continuity, as it boils down to verifying that the preimage of a cylinder is also a cylinder.
			\item \textbf{$G$-invariance:} By the same argument as in the proof of Theorem \ref{th:X_Htoshift} it suffices to show that for any $g\in G$ we have $gX_H^{G/H}\subset X_H^{G/H}$.
			
			Fix a transversal $M$. For any $g\in G$ by $h^g$ we denote the element of $H$ such that $g=h^gm^{gH}$, where $m^{gH}$ is the only element of $M\cap gH$. 
			
			Fix $g\in G$ and $x=\varphi_M\left(\left(x^\gamma\right)\right)\in X_H^{G/H}$. For any $\tilde{g}\in G$ we have
			$$h^{\tilde{g}g}=\tilde{g}g(m^{\tilde{g}gH})^{-1}=h^{\tilde{g}}m^{\tilde{g}H}g(m^{\tilde{g}gH})^{-1}.$$
			Therefore $m^{\tilde{g}H}g(m^{\tilde{g}gH})^{-1}\in H$ and
			\begin{multline*}
				gx(\tilde{g})=x(\tilde{g}g)=x\left(h^{\tilde{g}g}m^{\tilde{g}gH}\right)\\
				=x^{\tilde{g}gH}\left(h^{\tilde{g}g}\right)=x^{\tilde{g}H\cdot gH}\left(h^{\tilde{g}}m^{\tilde{g}H}g(m^{\tilde{g}gH})^{-1}\right)
				\\=\left(m^{\tilde{g}H}g(m^{\tilde{g}H\cdot gH})^{-1}\right)x^{\tilde{g}H\cdot gH}\left(h^{\tilde{g}}\right),
			\end{multline*}
			where $\tilde{g}H\cdot gH = \tilde{g}gH $ denotes the product of cosets in the quotient group $G/H$.
			
			Now, for each $\gamma\in G/H$ define: 
			$$y^\gamma:=\left(m^{\gamma}g(m^{\gamma\cdot gH})^{-1}\right)x^{\gamma\cdot gH}\in X_H.$$ 
			For any $\tilde g\in G$ we have 
			$$\varphi_M\left(\left(y^\gamma\right)\right)(\tilde g)=\left(m^{\tilde{g}H}g\left(m^{\tilde gH\cdot gH}\right)^{-1}\right)x^{\tilde gH\cdot gH}(h^{\tilde{g}})=gx(\tilde{g}).$$
			Hence we have $gx\in X_H^{G/H}$. Since $x$ was arbitrary, $gX_H^{G/H}\subset X_H^{G/H}$.
			
			\textbf{Compactness} is clear, because $X_H^{G/H}$ is an image of the compact set $(X_H)^{G/H}$ {under} the continuous function $\varphi_M$. 
			
		\end{enumerate}
	\end{proof}
	From now on $M$ is a fixed transversal for $H$ and $\varphi=\varphi_M$
	\begin{fact}
		\label{faktozawieraniu}
		$X\subset X_H^{G/H}$
	\end{fact}
	\begin{proof}
		Take any $x\in X$ and define $x^{Hm}:=(mx)|_H$ for any $m\in M$. Then $x=\varphi\left(\left( x^\gamma\right)\right)$ (we have $x^{Hm}(h)=mx(h)=x(hm)$).
	\end{proof}
	
	\begin{lemma}
		Let $F_1,\ldots,F_n\in \mathcal F(G)$ be subsets of different cosets. Then for $F=\bigcupdot_{i=1}^n F_i$ we have
		$$\mathcal L_F (X_H^{G/H})=\{\mathsf C\in \mathcal A^F: \forall_{i=1,\ldots,n}~ \mathsf C|_{F_i}\in \mathcal L_{F_i}(X_H^{G/H})\},$$
		and hence
		$$\left|\mathcal L_F (X_H^{G/H})\right|=\prod_{i=1}^n \left|\mathcal L_{F_i}(X_H^{G/H})\right|.$$
	\end{lemma}
	\begin{proof}
        It is obvious that the set on the right hand side contains $\mathcal L_F (X_H^{G/H})$. To see the opposite inclusion, take any configurations $\mathsf C_i\in \mathcal L_{F_i}(X_H^{G/H})$, $i=1,\ldots,n$, and let configuration $\mathsf C$ on $F$ be such that for any $i$ we have $\mathsf C|_{F_i}=\mathsf C_i$. For any $i$ there exists $x_i=\varphi((x_i^\gamma))\in X_H^{G/H}$ such that $x_i|_{F_i}=\mathsf C_i$. Let $m_i\in M$ be such that $F_i\subset Hm_i$. By the definition there exists $x\in X_H^{G/H}$ such that for any $i$ we have $(m_ix)|_{H}=x_i^{Hm_i}$, which is equivalent to $x|_{Hm_i}=x_i|_{Hm_i}$. Hence we get that
		$$\forall_{i=1,\ldots,n}~~ x|_{F_i}=x_i|_{F_i}=\mathsf C_i,$$
		so $\mathsf C=x|_F\in\mathcal L_F(X_H^{G/H})$.
	\end{proof}
	\noindent The proof of the following theorem was suggested by Tomasz Downarowicz.
	\begin{theorem}
		\label{thm:entropiaX_H^G/H}
		\[h(X_H^{G/H})=h(X_H)\]
	\end{theorem}
	
	\begin{proof}
		First we prove that {$h(X_H^{G/H})\leqslant h(X_H)$}. By using Theorem \ref{zasadainfimum} we get
		\begin{multline*}
			h(X_H^{G/H})=\inf_{F\in \mathcal F(G)}\frac{1}{|F|}\log\left|\mathcal{L}_F(X_H^{G/H})\right|\leqslant \inf_{F\in \mathcal F(H)}\frac{1}{|F|}\log\left|\mathcal{L}_F(X_H^{G/H})\right|\\
			=\inf_{F\in \mathcal F(H)}\frac{1}{|F|}\log\left|\mathcal{L}_F(X_H)\right|=h(X_H).
		\end{multline*}
		To prove $h(X_H^{G/H})\geqslant h(X_H)$, take any $F\in \mathcal F(G)$. Let $m_1,\ldots,m_n$ be all elements of $M$ such that $F_i:=F\cap Hm_i\neq \varnothing$ for $i=1,\ldots,n$. It is obvious that $F=\bigcupdot _{i=1}^nF_i$. By the preceding lemma and the definition of $X_H^{G/H}$ we have
		$$\left|\mathcal{L}_F(X_H^{G/H})\right|=\prod_{i=1}^n\left|\mathcal{L}_{F_i}(X_H^{G/H})\right|=\prod_{i=1}^n\left|\mathcal{L}_{F_im_i^{-1}}(X_H)\right|,$$
		so by using Theorem \ref{zasadainfimum} 
		\begin{multline*}
			\frac{1}{|F|}\log\left|\mathcal{L}_F(X_H^{G/H})\right|=\frac{1}{|F|}\sum_{i=1}^n\log\left|\mathcal{L}_{F_im_i^{-1}}(X_H)\right|\\
			=\frac{1}{\sum_{i=1}^n|F_i|}\sum_{i=1}^n|F_i|\frac{1}{|F_im_i^{-1}|}\log\left|\mathcal{L}_{F_im_i^{-1}}(X_H)\right|\\
			\geqslant \frac{1}{\sum_{i=1}^n|F_i|}\sum_{i=1}^n|F_i|h(X_H)=h(X_H).
		\end{multline*}
	\end{proof}
	
    Now we turn to the main question---when does the equality of the entropies of $X$ and $X_H$ imply $X=X_H^{G/H}$? In \cite{JKM} and \cite{S} the counterexamples are given that such implication fails in general. As it was done in these papers, in order to obtain a positive result we additionally assume a proper mixing condition.
	\begin{definition}
		Let $D\in\mathcal F(G)$ contain the neutral element of $G$. We say that $X$ is \textit{$D$-strongly irreducible}, if for any finite $B_1,B_2\subset G$ such that $DB_1\cap B_2=\varnothing$ we have
		$$\forall_{x,y\in X}\exists_{z\in X}~~z|_{B_1}=x|_{B_1} \land z|_{B_2}=y|_{B_2}.$$
	\end{definition}
 We remark that in case of multidimensional subshifts this condition is stronger than the uniform filling property, which was sufficient for results of \cite{S}.

    From now on $D\in\mathcal F(G)$ is fixed.
	
    In the course of proving the main theorem we will exploit the notion of the lower Banach density. We briefly recall its definition and formulate an auxiliary lemma. 

        \begin{definition}
		Let $A\subset G$. The \textit{lower Banach density} of $A$ is a number $$\underline{{d}}(A):=\sup_{F\in \mathcal F(G)}~\inf_{g\in G}\frac{|A\cap Fg|}{|F|}.$$
	\end{definition}
		It is known that $\underline d(A)=\lim_{n\to\infty}\inf_{g\in G} \frac{|A\cap F_n g|}{|F_n|}$, where $(F_n)$ is an arbitrary F{\o}lner sequence in $G$.
	
	\begin{lemma}
		\label{lematozbiorzeP}
		For any finite $B\subset G$ there exists $P\subset G$ such that:
		\begin{enumerate}
			\item $P$ has positive lower Banach density,
			\item for any $g,\tilde{g}\in P$, if $g\neq \tilde g$, then $Bg\cap B\tilde g=\varnothing.$
		\end{enumerate}
	\end{lemma}
	For the proof see e.g. \cite{FH} (Corollary 1.11).

	To prove the next theorem we adapt the ideas used in \cite{S}.
	\begin{theorem}
		\label{thm:jakirreducibletoentropiamniejsza}
		Let $X$ be $D$-strongly irreducible. Then for any proper subshift $Y\subsetneq X$ we have
		$h(Y)<h(X).$
	\end{theorem}
	\begin{proof}
		Since $Y$ is a proper subshift, there exists, by Fact \ref{fact_zawieranie_jezykow}, a configuration $\mathsf C\in \mathcal L(X)\setminus \mathcal L(Y)$ on some $B\in\mathcal F(G)$. Take $\tilde B:=DB$. From Lemma \ref{lematozbiorzeP} there exists $P\subset G$ such that $\underline d(P)>0$ and for any different $g,\tilde g\in P$ we have $\tilde Bg\cap \tilde B \tilde g=\varnothing$.
		
		Let $(F_n)$ be a F{\o}lner sequence in $G$. We know that $$\underline d(P)=\lim_{n\rightarrow \infty}\inf_{g\in G}\frac{|P\cap F_ng|}{|F_n|}.$$ Therefore for $n$ greater than some $N$ we have
		$$\frac{|P\cap F_n|}{|F_n|}\geqslant \inf_{g\in G}\frac{|P\cap F_ng|}{|F_n|}\geqslant \underline d(P)- \frac{\underline d(P)}2= \frac{\underline d(P)}2>0.$$
		Hence for $n>N$ we have
		\begin{equation}
			\label{onestar}
			|P\cap F_n|\geqslant\frac{\underline d(P)}2 |F_n|.
		\end{equation}
		Now, define $\tilde F_n:=\tilde B F_n$. By Lemma \ref{lematopowiekszaniuCF} $(\tilde F_n)$ is a F{\o}lner sequence in $G$. Fix $n>N$ and take $J:=P\cap F_n$ (notice, that for any $g\in J$ we have $\tilde Bg\subset \tilde F_n$). For any $I\subset J$ we define 
		$$\mathcal L_n^I:=\left\{x|_{\tilde F_n}: x\in X \land \forall_{g\in J\setminus I}\textrm{~} (gx)|_{B}\neq \mathsf C\right\}.$$
		In particular, $\mathcal L_n^J=\mathcal L_{\tilde F_n}(X)$ and $\mathcal L_n^{\varnothing} =\left\{x|_{\tilde F_n}: x\in X \land \forall_{g\in J}\textrm{~} (gx)|_{B}\neq \mathsf C\right\}$.
		
		Let us enumerate elements of $J$: $J=\{g_1\textrm{, }, g_2\textrm{, },\ldots,g_{|J|}\}$. For any $I=\{g_1,\ldots,g_i\}$, where $i<|J|$, we have
		$$\left|\mathcal L_n^I\right|=\left|\mathcal L_n^{I\cup\{g_{i+1}\}}\setminus\left\{x|_{\tilde F_n}: x\in X \land x|_{\tilde F_n}\in\mathcal L_n^{I\cup\{g_{i+1}\}}\land (g_{i+1}x)|_{B}=\mathsf C\right\} \right|$$
		\begin{equation}
			\label{twostars}
			=\left|\mathcal L_n^{I\cup\{g_{i+1}\}}\right| - \left|\left\{x|_{\tilde F_n}: x\in X \land x|_{\tilde F_n}\in\mathcal L_n^{I\cup\{g_{i+1}\}}\land (g_{i+1}x)|_{B}=\mathsf C\right\} \right|
		\end{equation}
        (we allow $I$ to be empty, then we take $i=0$).
		
		On the other hand, for any nonempty $I\subset J$ and $g\in I$ we have
		$$\left|\mathcal L_n^I\right|
		\leqslant \left|\mathcal L_{\tilde Bg}(X)\right| \left|\left\{x|_{\tilde F_n\setminus \tilde Bg}: x\in X \land x|_{\tilde F_n}\in \mathcal L_n^I \right\}\right|$$
		$$\leqslant \left|\mathcal L_{\tilde B}(X)\right| \left|\left\{x|_{\tilde F_n}: x\in X \land x|_{\tilde F_n}\in \mathcal L_n^I \land (gx)|_{B}=\mathsf C\right\}\right|,$$
		where the second inequality comes from the facts that $X$ is $D$-strongly irreducible and that $\left|\mathcal L_{\tilde Bg}(X)\right|=\left|\mathcal L_{\tilde B} (X)\right|.$ Hence for any nonempty $I\subset J$ and any $g\in I$ we get
		$$\left|\left\{x|_{\tilde F_n}: x\in X \land x|_{\tilde F_n}\in \mathcal L_n^I \land x|_{Bg}=\mathsf C\right\}\right|\geqslant\left|\mathcal L_{\tilde B}(X)\right|^{-1} \left|\mathcal L_n^I\right|.$$
		From \eqref{twostars} and the above we obtain
		\begin{equation}
			\label{threestars}
			\forall _{I=\{g_1,\ldots,g_i\}\textrm{, } i<|J|}~~\left|\mathcal L_n^I\right|\leqslant \left|\mathcal L_n^{I\cup\{g_{i+1}\}}\right|\left(1-\left|L_{\tilde B} (X)\right|^{-1}\right).
		\end{equation}
		Now, let us notice that $\mathcal L _{\tilde F_n}(Y)\subset\mathcal L_n^{\varnothing}$. Therefore, by using \eqref{threestars}
		$$\left|\mathcal L _{\tilde F_n}(Y)\right|\leqslant \left|\mathcal L_n^{\varnothing}\right|\leqslant \left|\mathcal L_n^{\{g_{1}\}}\right|\left(1-\left|L_{\tilde B} (X)\right|^{-1}\right).$$
		Inductively, by using \eqref{threestars} $\left(|J|-1\right)$ times more, we get
		$$\left|\mathcal L _{\tilde F_n}(Y)\right|\leqslant \left|\mathcal L_n^J\right|\left(1-\left|L_{\tilde B} (X)\right|^{-1}\right)^{|J|}$$
		$$\leqslant \left|\mathcal L_n^J\right|\left(1-\left|L_{\tilde B} (X)\right|^{-1}\right)^{\frac{\underline d(P)}2 |F_n|}.$$
		The last inequality comes from \eqref{onestar} and the fact that $\left(1-\left|L_{\tilde B} (X)\right|^{-1}\right)\in (0,1)$.
		Hence we get
		\begin{multline*}
			\frac{1}{\left|\tilde F_n\right|}\log{\left|\mathcal L _{\tilde F_n}(Y)\right|}\leqslant \frac{1}{\left|\tilde F_n\right|} \log{\left|\mathcal L_n^J\right|} + \frac{\frac{\underline d(P)}{2}|F_n|}{\left|\tilde B F_n\right|}\log{\left(1-\left|L_{\tilde B} (X)\right|^{-1}\right)}\\
			\leqslant \frac{1}{\left|\tilde F_n\right|} \log{\left|\mathcal L_n^J\right|} + \frac{\frac{\underline d(P)}{2}|F_n|}{\left|\tilde B\right| \left| F_n\right|}\log{\left(1-\left|L_{\tilde B} (X)\right|^{-1}\right)}\\
			=\frac{1}{\left|\tilde F_n\right|} \log{\left|\mathcal L_{\tilde F_n}(X)\right|} + \frac{\frac{\underline d(P)}{2}}{\left|\tilde B\right|}\log{\left(1-\left|L_{\tilde B} (X)\right|^{-1}\right)}.
		\end{multline*}
		By letting $n$ go to infinity ($n$ was fixed, greater than $N$) we obtain
		$$h(Y) \leqslant h(X) + \frac{\underline d(P)}{2\left|\tilde B\right|}\log{\left(1-\left|L_{\tilde B} (X)\right|^{-1}\right)}<h(X).$$
	\end{proof}
	
	\begin{theorem}
		\label{lematzeX_H^G/Hjestirreducible}
		If $X$ is $D$-strongly irreducible, then so is $X_H^{G/H}$.
	\end{theorem}
	
	\begin{proof}
		Let $B_1,B_2\subset G$ be finite and such that $DB_1\cap B_2=\varnothing$. Take any $x,y\in X_H^{G/H}$ and let $(x^\gamma),(y^\gamma)\in (X_H)^{G/H}$ be such that $x=\varphi((x^\gamma))$ and $y=\varphi((y^\gamma))$. Write $M=\{m_1,m_2,\ldots\}$. For any $i=1,2,\ldots$ define $S_i=B_1\cap Hm_i$ and $T_i=B_2\cap Hm_i$.
		
		By the definition of $\varphi$ for each $i$ we have $(m_ix)|_{S_im_i^{-1}}=x^{Hm_i}|_{S_im_i^{-1}}$ and $(m_iy)|_{T_im_i^{-1}}=y^{Hm_i}|_{T_im_i^{-1}}$ (since $S_i$ and $T_i$ are subsets of $Hm_i$ and $H$ is normal, $S_im_i^{-1}$ and $T_im_i^{-1}$ are subsets of $H$). Moreover, there exist $x_i, y_i\in X$ such that $x^{Hm_i}=x_i|_H$ and $y^{Hm_i}=y_i|_H$. Since $X$ is $D$-strongly irreducible
		and 
		$$DS_im_i^{-1}\cap T_i m_i^{-1}\subset (DB_1\cap B_2)m_i^{-1}=\varnothing,$$
		there exists $z_i\in X$ such that $z_i|_{S_im_i^{-1}}= x_i|_{S_im_i^{-1}}$ and $ z_i|_{T_im_i^{-1}}= y_i|_{T_im_i^{-1}}$. By taking $z^{Hm_i}= z_i|_H$ for any $i$ and defining $z=\varphi((z^\gamma))$ we get $z|_{S_i}=x|_{S_i}$ and $z|_{T_i}=y|_{T_i}$ for any $i$. Therefore $z|_{B_1}=|x_{B_1}$ and $z|_{B_2}=|y_{B_2}$.
	\end{proof}
	We are now ready to formulate the main theorem.
	\begin{theorem}
		\label{guwnetwerdzenie}
		If $X$ is $D$-strongly irreducible and $h(X)=h(X_H)$, then $X=X_H^{G/H}$.
	\end{theorem}
	\begin{proof}
		From Fact \ref{faktozawieraniu} we know that $X\subset X_H^{G/H}$. Assume that $X\subsetneq X_H^{G/H}$. By Theorem \ref{lematzeX_H^G/Hjestirreducible} $X_H^{G/H}$ is $D$-strongly irreducible. Then by theorems \ref{thm:jakirreducibletoentropiamniejsza} and \ref{thm:entropiaX_H^G/H} we get $h(X)<h(X_H)$.
	\end{proof}

        \textbf{Aknowledgements.} First of all I want to thank my supervisor Bartosz Frej, who introduced me into the subject and was helping me during my work. I am also grateful to Tomasz Downarowicz for lots of useful advice.
	

\begin{thebibliography}{JKM}
		
		\bibitem{JKM} A. Johnson, S. Kass and K. Madden, \emph{Projectional entropy in higher dimensional shifts of finite type.} Complex Systems, 17 (2007), 243–257.
		
		\bibitem{S} M. Schraudner, \emph{Projectional entropy and the electrical wire shift.} Discrete and Continuous Dynamical Systems, 26 (2010), 333–346.
  
		\bibitem{LW} E. Lindenstrauss and B. Weiss, \emph{Mean topological dimension}, Israel J. Math., 115 (2000), 1–24.
		
		\bibitem{DFR} T. Downarowicz, B. Frej, and P. P. Romagnoli, \emph{Shearer’s inequality and infimum rule for Shannon entropy and topological entropy}, Dynamics and numbers, Contemp. Math., vol. 669, Amer. Math. Soc., Providence, RI, 2016, pp. 63–75.
		
		\bibitem{FH} B. Frej and D. Huczek, \emph{Minimal models for actions of amenable groups}, Groups Geom. Dyn. 11 (2017), 567–583.
	\end{thebibliography}
\end{document}